\documentclass[11pt]{amsart}
\usepackage[T1]{fontenc}
\usepackage[a4paper]{anysize}\marginsize{2.0cm}{2.0cm}{2.0cm}{2.0cm}
\pdfpagewidth=\paperwidth \pdfpageheight=\paperheight
\usepackage{pgf,tikz,hyperref}
\usetikzlibrary{arrows}
\usepackage{amsfonts,amssymb,amsthm,amsmath,eucal}
\usepackage{graphicx}
\usepackage{ltablex}
\usepackage{caption, booktabs}
\usepackage{longtable}

\theoremstyle{plain}
\newtheorem{thm}{Theorem}[section]
\newtheorem{theorem}[thm]{Theorem}

\newtheorem{lemma}[thm]{Lemma}
\newtheorem{proposition}[thm]{Proposition}
\newtheorem{corollary}[thm]{Corollary}

\theoremstyle{definition}
\newtheorem{definition}[thm]{Definition}
\newtheorem{remark}[thm]{Remark}
\newtheorem{example}[thm]{Example}

\newtheorem{notation}[thm]{Notation}

\newtheorem{thevarthm}[thm]{\varthmname}

\newenvironment{varthm*}[1]{\trivlist\item[]{\bf #1.}\it}{\endtrivlist}

\newcommand\be{\begin{eqnarray*}}
\newcommand\ee{\end{eqnarray*}}

\newcommand\C{\mathbb C}
\DeclareMathOperator{\GL}{GL}

\newcommand\Z{\mathbb Z}

\newcommand\R{\mathbb R}
\newcommand\cH{\mathcal H}
\newcommand\cB{\mathcal B}
\newcommand\cD{\mathcal D}
\newcommand\cF{\mathcal F}

\renewcommand\H{\mathbb H}

\renewcommand\P{\mathbb P}

\newcommand\newop[2]{\def#1{\mathop{\rm #2}\nolimits}}
\newop\edim{edim}
\newop\Zeroes{Zeroes}
\newop\Jac{Jac}
\newop\Ass{Ass}
\newop\SL{SL}
\newop\PGL{{\P}GL}
\newop\Km{Km}
\newop\reg{reg}
\newop\mult{mult}
\newop\alphahat{\widehat{\alpha}}

\newcommand\eqnref[1]{(\ref{#1})}

\newcolumntype{L}{>{$}l<{$}}

\renewcommand\keywords[1]{{\renewcommand\thefootnote{}\footnotetext{\textit{Keywords:} #1.}}}
\newcommand\subclass[1]{{\renewcommand\thefootnote{}\footnotetext{\textit{Mathematics Subject Classification (2010):} #1.}}}

\def\endproof{\hspace*{\fill}\endproofsymbol\endtrivlist}

\def\endproofsymbol{\frame{\rule[0pt]{0pt}{6pt}\rule[0pt]{6pt}{0pt}}}

\begin{document}

\author[S. Calvo]{Sebastian Calvo}
\address{Department of Mathematics, Towson University, Towson, MD 21252}
\email{scalvo@towson.edu}

\author[J. Huizenga]{Jack Huizenga}
\address{Department of Mathematics, The Pennsylvania State University, University Park, PA 16802}
\email{huizenga@psu.edu}

\author[T. Szemberg]{Tomasz Szemberg}
\address{Department of Mathematics, University of the National Education Commission, PL-30-084 Krakow, Poland}
\email{tomasz.szemberg@gmail.com}

\title{Waldschmidt constants of symmetric sets of points in $\P^3$}
\date{\today}

\thispagestyle{empty}

\begin{abstract}
Configurations of points defined by complex reflection groups have attracted a lot of attention recently in several directions of research, e.g., the containment problem between ordinary and symbolic powers of ideals, in the theory of unexpected hypersurfaces, in the study of sets of points whose general projections are complete intersections and in the ideas revolving around the Bounded Negativity Conjecture. In order to understand these configurations better several attempts have been undertaken to compute their various invariants. In this paper we focus on their Waldschmidt constants. In the case of plane configurations of points determined by reflection groups most (but not all) Waldschmidt constants are known. Here we pass to configurations in $\P^3$ where new ideas are required as the identification between divisors and curves is no longer available.  In particular, we precisely compute the Waldschmidt constant for configurations of points in $\P^3$ coming from the $D_4,B_4,F_4$, and $H_4$ root systems.
\end{abstract}
\keywords{linear series, reflection group, Waldschmidt constant}
\subclass{14C20, 14N20, 13A50, 52C35}
\maketitle
%*****************************************************************************
\section{Introduction}
Finite sets of points have been intensively studied ever since the inception of geometry. A careful systematic treatment of both planar and three dimensional geometry by Euclid \cite{Euclid} culminated in the classification of Platonic solids. Their vertices form remarkably symmetric sets of points in space. The introduction of algebraic methods to geometry by Ren\'e Descartes in the 17th century \cite{Descartes} led eventually to the creation of algebraic geometry and created the possibility of attaching  various algebraic invariants to finite sets of points in particular. One of these invariants which has recently attracted a lot of attention is the Waldschmidt constant, see \cite{Calvo2024}, \cite{DSS2024}, \cite{BisuiNguyen2024}, \cite{Many2022}, \cite{ChangJow2020}, \cite{BDHHSS2019}. It can be defined for arbitrary non-trivial ideals in a graded ring, but we restrict here our attention to ideals of points, where it can be stated in a considerably simpler way than in the general set-up thanks to the Zariski-Nagata Theorem, see \cite[Theorem 3.14]{Eisenbud}.
\begin{definition}[Waldschmidt constant of a finite set of points]\label{def: WC} Let $Z\subset \P^N$ be a finite set of points and let $I=I(Z) \subset S = \C[x_0,\ldots,x_N]$ be its homogeneous ideal, so
$$I(Z)=\bigcap_{P\in Z}I(P).$$
We write $$I^{(m)}=\bigcap_{P\in Z}(I(P))^m$$ for the $m$th symbolic power of $I$, 
and $\alpha(I^{(m)})$ for the least degree of a non-zero polynomial in $I^{(m)}$.  Then the \emph{Waldschmidt constant of $Z$} is 
\begin{equation}\label{eq: WC}
    \alphahat(Z)=\inf_{m\in\Z_{+}}\frac{\alpha(I^{(m)})}{m},
\end{equation}   

\end{definition}
\begin{remark}
It is well known that the Waldschmidt constant is actually the limit of the sequence in \eqref{eq: WC} and that there is a universal upper bound
$$\alphahat(I)\leq\sqrt[N]{\#Z}$$
valid in the setting of Definition \ref{def: WC}.
\end{remark}

\subsection*{Symmetric sets of points} In \cite{BDHHSS2019} the authors studied Waldschmidt constants of sets of points arising as singular loci of certain extreme line arrangements in $\P^2$. More specifically the Klein and the Wiman arrangements. These configurations arose from the action of certain complex reflection groups on $\P^2$. The group action can be exploited to greatly simplify the computation of the Waldschmidt constant of a collection of points which is preserved by the group action. Note however that this is not sufficient to handle all details, in particular the Waldschmidt constant of the Klein configuration remains unknown, see \cite{BDHHSS2019}. A particularly nice aspect of the computation of Waldschmidt constants on surfaces is that hypersurfaces are curves.  This makes it so that the computation of the Waldschmidt constant is closely connected to the study of nef divisor classes on the surface blown up at the point configuration.

Here we turn our attention to several highly symmetric sets of points in $\P^3$.  In contrast with the surface case, the fact that divisors are now surfaces instead of curves typically makes the computation of the Waldschmidt constant a much more challenging problem.  However, we will see that the presence of a large group of symmetries puts the problem within reach.

\subsection*{The $\cD_4$, $\cB_4$, and $\cF_4$ configurations} The point configurations in this paper all arise as the projectivizations of root systems in $\R^4$.  We first consider three related sets of points in $\P^3$.  We have root systems $$D_4\subset B_4\subset F_4 \subset \R^3$$ with respectively $24, 32$, and $48$ vectors.  For example, the root system $D_4$ consists of the vertices of a regular $24$-cell.  Projectivizing these collections, we obtain sets $$\cD_4\subset \cB_4\subset \cF_4\subset \P^3$$ with $12, 16,$ and $24$ points.  

\begin{theorem}\label{thm: main}
We have $$\alphahat(\cD_4)=\alphahat(\cB_4)=2 \quad \textrm{and} \quad \alphahat(\cF_4) = \frac{8}{3}.$$
\end{theorem}

In each of these cases, there is a natural configuration of hyperplanes containing the point configuration.  The product of their equations computes the Waldschmidt constant.  The challenging part of the argument, especially in the case of $\cF_4$, is a proof that there are no hypersurfaces computing a smaller Waldschmidt ratio.  The definitions of these configurations and the computation of their Waldschmidt constants will be undertaken in Sections \ref{sec: D4}, \ref{sec: B4}, and \ref{sec: F4}.

\subsection*{The $\cH_4$ configuration} The root system $H_4\subset \R^4$ is a collection of 120 vectors in $\R^4$ which can be viewed as the vertices of a regular $600$-cell.  Its projectivization $\cH_4\subset \P^3$ is a collection of $60$ points.  In contrast with previous cases, the Waldschmidt constant is computed by a very interesting hypersurface, and we need new methods to compute it.

There is a complex reflection group $G$ of order $14400$ acting on $\P^3$, and the points in $\cH_4$ form a single orbit under this group action.  One of these points is the point $(0:0:0:1)$ with homogeneous ideal $I = (x,y,z)$.  By an ``averaging'' process, it is easy to see that the Waldschmidt constant of $\cH_4$ can be computed by only considering $G$-invariant polynomials.   Since $G$ is a complex reflection group, the ring $T$ of $G$-invariant polynomials is itself a polynomial algebra $$T = S^G = \C[f_2,f_{12},f_{20},f_{30}] \subset S$$ generated by algebraically independent polynomials of degrees $2,12,20,30$.  There is then an invariant polynomial of degree $d$ having multiplicity exactly $m$ at all the points in $\cH_4$ precisely when the vector space $$\frac{T_d \cap I^m}{T_d \cap I^{m+1}}$$ is nonzero.  These vector spaces can be packaged together into a ring $$R = \bigoplus_{m\geq 0} \frac{T\cap I^m}{T\cap I^{m+1}}$$ which is bigraded by both the degree $d$ and multiplicity $m$.  Thus the Waldschmidt constant can be computed as $$\alphahat(\cH_4) = \inf \left\{\frac{d}{m}:R_{d,m}\neq 0 \right\}.$$  By working explicitly with the group $G$, we are able to show the following.
\begin{theorem}\label{thm: intro-H4}
The ring $R$ is finitely generated by a list of $12$ explicit generators.  Among the generators, the generator of bidegree $(36,10)$ has the smallest ratio $\frac{d}{m}$.  Therefore, $$\alphahat(\cH_4)=\frac{36}{10}.$$ Furthermore, the bigraded Hilbert series of $R$ can be explicitly computed, so the dimensions of the graded pieces $R_{d,m}$ can be efficiently determined.
\end{theorem}

The proof of the theorem will occupy Section \ref{sec: H4}.

\begin{remark}
In other settings it is easy to make analogous constructions of rings $R$ where the ring is at least conjectured to not be finitely generated.  For example, if we use the full polynomial ring $S=\C[x,y,z]$ and let $I$ be the ideal of $10$ very general points in $\P^2$, we can form the ring $$R = \bigoplus_{m\geq 0} \frac{S \cap I^{(m)}}{S\cap I^{(m+1)}}.$$ (Note we have used symbolic powers here, although in the previous setting there was no difference between symbolic and ordinary powers.)  The Nagata conjecture would imply that $$\inf \left\{\frac{d}{m} : R_{d,m}\neq 0\right\} = \sqrt{10},$$ which is inconsistent with $R$ being finitely generated.  A similar example would be to consider a single point on a quintic surface in $\P^3$, where the conjectured infimum is similarly irrational.  Thus, the finite generation in Theorem \ref{thm: intro-H4} is quite surprising, and very particular to the $\cH_4$ configuration of points.
\end{remark}

\subsection*{Geproci sets of points}  Our particular interest in the sets considered in this paper was sparked by recent developments in the study of \emph{geproci} sets of points; see \cite{POLITUS1} for introduction to this circle of ideas.
\begin{definition}[Geproci set of points]\label{def: geproci}
A finite set of points $Z\subset\P^3$ is an {\it $(a,b)$-geproci set} if its projection from a general point in $\P^3$ to a plane $H\subset \P^3$ is the transverse intersection of curves in $H$ of degrees $a$ and $b$. 
\end{definition}
The sets $\cD_4\subset \cB_4\subset \cF_4$ and $\cH_4$ are all both geproci and projectivizations of root systems.  In addition they have the following curious properties:
\begin{itemize}
    \item[$\cD_4$] is a $(3,4)$-geproci set and it has the least number of points among all geproci sets in $\P^3$;
    \item[$\cB_4$] is a $(4,4)$-geproci set, even though it contains $\cD_4$ as a proper subset, they have both the same Waldschmidt constant;
    \item[$\cF_4$] is a $(4,6)$-geproci set, whose points split into two orbits with respect to the associated reflection group;
    \item[$\cH_4$] is a $(6,10)$-geproci set, where none of the curves in Definition \ref{def: geproci} can be taken as a union of lines.
\end{itemize}

\section{The $D_4$ configuration}\label{sec: D4}
The root system $D_4$ consists of $24$ points in $S^3\subset \R^4$, which can be thought of as the vertices of a regular $24$-cell. Their coordinates can be taken as permutations of
$$(\pm 1,\pm 1,0,0).$$
Projectivizing we obtain the set $\cD_4$ of $12$ points in $\P^3$ with coordinates
$$(\pm 1:1:0:0)$$
and permutations thereof. 

In each of the coordinate planes there are six points from $\cD_4$ and they form there a \emph{star configuration}, i.e., they are the intersection points of pairs of $4$ lines. By \cite[Proposition 2.9]{GHM2013} the Waldschmidt constant of their ideal considered as an ideal in the ring of the remaining $3$ variables equals $2$. In particular, any surface $S$ in $\P^3$ either contains all four coordinate planes or its restriction to any of them is a curve whose degree is at least twice the least multiplicity of $S$ in the points of $\cD_4$. It follows that
$$\alphahat(\cD_4)\geq 2.$$
On the other hand, since the points in $\cD_4$ are equidistributed on coordinate planes, each of them is contained in exactly $2$ of the planes. Taking the union of these planes, we obtain a divisor of degree $4$ vanishing to order $2$ in all points of $\cD_4$, hence
$$\alphahat(\cD_4)\leq 2$$
and we are done with the proof of Theorem \ref{thm: main} in this case.

\section{The $B_4$ configuration}\label{sec: B4}
The root system $B_4$ is the extension of $D_4$ by the $8$ points whose coordinates are permutations of 
$$(\pm 1,0,0,0).$$
The projectivized set $\cB_4$ consists of $16$ points, which are the union of $\cD_4$ and the four fundamental points (where all but one coordinate are $0$). Thus
$$\alphahat(\cB_4)\geq \alphahat(\cD_4)=2.$$
Taking again the union of the four coordinate planes we obtain a divisor of degree $4$ vanishing in all points of $\cD_4$ to order $2$ and in the fundamental points to order $3$. Hence
$$\alphahat(\cB_4)\leq 2$$
and we are done with the proof of Theorem \ref{thm: main} in this case.

\section{The $F_4$ configuration}\label{sec: F4}
As a root system $F_4$ consists of $48$ vectors in $\R^4$, which split into two sets of vertices of two $24$-cells. One is $D_4$ as described in section \ref{sec: D4} and the other one is defined by the eight roots resulting from permuting coordinates of $(\pm 1,0,0,0)$ and $16$ roots of the form $(\pm 1/2,\pm 1/2,\pm 1/2,\pm 1/2)$. This second set is projectively equivalent to $D_4$.

For the purpose of this work we look at the points projectively, so that there are $24$ of them. We enumerate them as follows:
\renewcommand{\arraystretch}{1.15}
$$\begin{array}{llll}
P_{1} = (1:1:0:0), & P_{2} = (1:-1:0:0), &
P_{3} = (1:0:1:0), & P_{4} = (1:0:-1:0),\\
P_{5} = (1:0:0:1), & P_{6} = (1:0:0:-1),&
P_{7} = (0:1:1:0), & P_{8} = (0:1:-1:0),\\
P_{9} = (0:1:0:1), & P_{10} = (0:1:0:-1),&
P_{11} = (0:0:1:1), & P_{12} = (0:0:1:-1),\\
P_{13} = (1:0:0:0), & P_{14} = (0:1:0:0),&
P_{15} = (0:0:1:0), & P_{16} = (0:0:0:1),\\
P_{17} = (1:1:1:1), & P_{18} = (1:1:-1:-1),&
P_{19} = (1:-1:1:-1), & P_{20} = (1:-1:-1:1),\\
P_{21} = (-1:1:1:1), & P_{22} = (1:-1:1:1),&
P_{23} = (1:1:-1:1), & P_{24} = (1:1:1:-1).
\end{array}$$
The Weyl group $G:=W(F_4) \subset \PGL(\R^4)$ is the subgroup generated by reflections orthogonal to the roots projecting to the points $P_i$.  It is the finite complex reflection group $G_{28}$ of order 1,152 in the Shephard-Todd classification.  It can be generated by the following matrices:
$$\begin{pmatrix}1 & 0 & 0 & 0\\ 0 & 0 & 1 & 0 \\ 0 & 1 & 0 & 0 \\ 0 & 0 & 0 & 1\end{pmatrix}, \quad 
\begin{pmatrix}1 & 0 & 0 & 0\\ 0 & 1 & 0 & 0 \\ 0 & 0 & 0 & 1 \\ 0 & 0 & 1 & 0\end{pmatrix}, \quad 
\begin{pmatrix} 1 & 0 & 0 & 0\\ 0 & 1 & 0 & 0\\ 0 & 0 & 1 & 0\\ 0 & 0 & 0 & -1 \end{pmatrix},
\quad \frac{1}{2} \begin{pmatrix} 1 & 1 & 1 & 1\\ 1 & 1 & -1 & -1\\ 1 & -1 & 1 & -1 \\ 1 & -1 & -1 & 1\end{pmatrix},$$
which correspond to the root basis consisting of $P_8, P_{12}, P_{16}$ and $P_{21}$.  The 24 points of the configuration split into two orbits of size 12 under the action of $G$.  The points $P_1,\dots, P_{12}$ inherited from $D_4$ lie in one orbit while the points $P_{13},\dots,P_{24}$ lie in a second orbit, which is projectively equivalent to the $D_4$ configuration as well.

For a point $Q=(a:b:c:d)\in\P^3$ we define the dual plane $H_Q$ by the equation $ax+by+cz+dw=0$.
It is easy to check that the $24$ planes dual to the above $24$ points form an arrangement of planes such that exactly $9$ of them contain any given point $P_i$. This means that their union $\H$ is a divisor of degree $24$ sitting in the $9$th symbolic power of the ideal $I(F_4)$ of the configuration of points. It follows that
$$\alphahat(\cF_4)\leq\frac{24}{9}=\frac{8}{3}.$$
The rest of this section is devoted to the proof of the reverse inequality.

Our approach in this part is geometric. We will show that any divisor $S$ of degree $d$ vanishing to order $m$ at $\cF_4$ such that $d/m<8/3$ must contain $\H$. Then the residual divisor $S'=S-\H$ has degree $d-24$ and vanishes to order $m-9$ at $\cF_4$. But we have again
$$\frac{d-24}{m-9}<\frac{8}{3},$$
so that $S'$ again contains $\H$ and the process continues for ever, which is of course not possible.

In the sequel collinearities of points in the configuration are essential. In particular, we have the following, easy to check, fact.
\begin{lemma}\label{lem: three points lines}
There are $32$ lines containing exactly $3$ points of $\cF_4$. The aligned triples are the following:    
$$\{1,3,8\},
\{1,4,7\},
\{1,5,10\},
\{1,6,9\},
\{2,3,7\},
\{2,4,8\},
\{2,5,9\},
\{2,6,10\},
\{3,5,12\},
\{3,6,11\},
\{4,5,11\},$$$$
\{4,6,12\},
\{7,9,12\},
\{7,10,11\},
\{8,9,11\},
\{8,10,12\},
\{13,17,21\},
\{13,18,22\},
\{13,19,23\},
\{13,20,24\},$$$$
\{14,17,22\},
\{14,18,21\},
\{14,19,24\},
\{14,20,23\},
\{15,17,23\},
\{15,18,24\},
\{15,19,21\},
\{15,20,22\},$$$$
\{16,17,24\},
\{16,18,23\},
\{16,19,22\},
\{16,20,21\}.
$$
\end{lemma}

Now, we focus on the plane $H=H_{16}=\left\{w=0\right\}$. This plane contains $9$ points from $\cF_4$ with numbers $\left\{1,2,3,4,7,8,13,14,15\right\}$. The alignments between these points are indicated in Figure \ref{fig: w=0 plane}. In particular the points with indices in the set $X=\left\{1,2,3,4,7,8\right\}$, coming from the first orbit, are intersection points of $4$ lines, so they form a star configuration. We denote the union of these $4$ lines by $\sigma$, and draw these lines as solid lines in Figure \ref{fig: w=0 plane}. The remaining points with indices in the set $Y=\left\{13,14,15\right\}$, coming from the other orbit, are vertices of a triangle, whose sides (the dashed lines) each contain two points from the aforementioned set. We denote the union of these three lines by $\tau$ and draw these lines as dashed lines in Figure \ref{fig: w=0 plane}. Thus the points in $X$ each lie on one line from $\tau$, and the points in $Y$ are the intersection points of pairs of lines in $\tau$.

\definecolor{uuuuuu}{rgb}{0.26666666666666666,0.26666666666666666,0.26666666666666666}
\definecolor{xdxdff}{rgb}{0.26666666666666666,0.26666666666666666,0.26666666666666666}
\definecolor{ududff}{rgb}{0.26666666666666666,0.26666666666666666,0.26666666666666666}
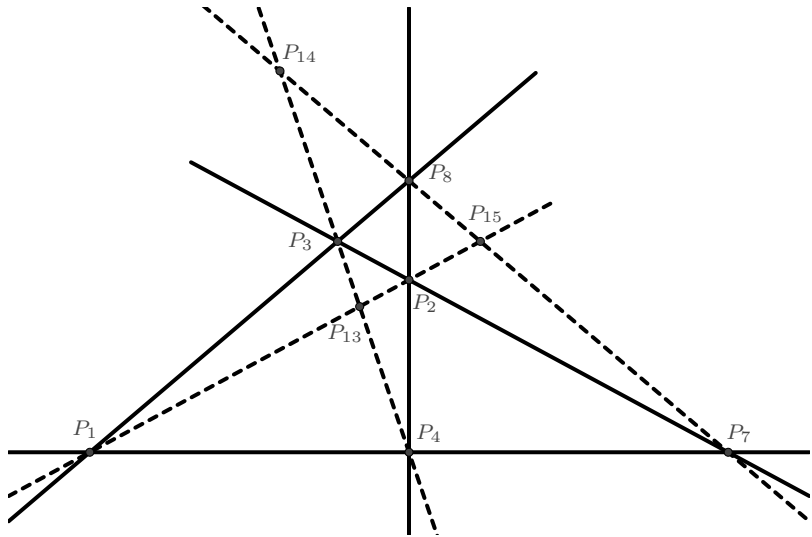
\begin{figure}[h!]
    \centering
\begin{tikzpicture}[line cap=round,line join=round,>=triangle 45,x=1cm,y=1cm,scale=0.6]
\clip(-8.78,-5.82) rectangle (8.78,5.82);
\draw [line width=1.5pt,domain=-8.78:8.78] plot(\x,{(-28-0*\x)/7}); %line 1 4 7
\draw [line width=1.5pt,domain=-8.78:2.78] plot(\x,{(--14--6*\x)/7}); %line 1 3 8
\draw [dashed,line width=1.5pt,domain=-8.78:8.78] plot(\x,{(--6.258823529411764--4.658823529411765*\x)/-1.564705882352941}); %line 3 4 13 14
\draw [line width=1.5pt,domain=-4.78:8.78] plot(\x,{(-1.6470588235294086-4.658823529411765*\x)/8.564705882352941}); %line 2 3 7
\draw [line width=1.5pt] (0,-5.82) -- (0,5.82); %line 2 4 8
\draw [dashed,line width=1.5pt,domain=-8.78:3.1] plot(\x,{(-1.3461538461538436--3.807692307692308*\x)/7}); %line 1 2 13 15
\draw [dashed,line width=1.5pt,domain=-8.78:8.78] plot(\x,{(-14--6*\x)/-7}); %line 7 8 14 15
\begin{scriptsize}
\draw [fill=ududff] (-7,-4) circle (2.5pt);
\draw[color=ududff] (-7.1,-3.49) node {$P_1$};
\draw [fill=xdxdff] (0,-4) circle (2.5pt);
\draw[color=xdxdff] (0.44,-3.57) node {$P_4$};
\draw [fill=xdxdff] (0,2) circle (2.5pt);
\draw[color=xdxdff] (0.7,2.17) node {$P_8$};
\draw [fill=xdxdff] (-1.564705882352941,0.658823529411765) circle (2.5pt);
\draw[color=xdxdff] (-2.38,0.66) node {$P_3$};
\draw [fill=xdxdff] (7,-4) circle (2.5pt);
\draw[color=xdxdff] (7.24,-3.57) node {$P_7$};
\draw [fill=uuuuuu] (0,-0.19230769230769193) circle (2.5pt);
\draw[color=uuuuuu] (0.35,-0.7) node {$P_2$};
\draw [fill=uuuuuu] (-2.829787234042552,4.425531914893616) circle (2.5pt);
\draw[color=uuuuuu] (-2.4,4.81) node {$P_{14}$};
\draw [fill=uuuuuu] (1.5647058823529407,0.658823529411765) circle (2.5pt);
\draw[color=uuuuuu] (1.68,1.27) node {$P_{15}$};
\draw [fill=uuuuuu] (-1.0813008130081299,-0.7804878048780484) circle (2.5pt);
\draw[color=uuuuuu] (-1.4,-1.37) node {$P_{13}$};
\end{scriptsize}
\end{tikzpicture}    
    \caption{$\cF_4\cap H_{16}$.  The curve $\sigma$ is the union of the four solid lines, and $\tau$ is the union of the three dashed lines.}
    \label{fig: w=0 plane}
\end{figure}

The last 4 lines from Lemma \ref{lem: three points lines} intersect the plane $H$ in the following 4 points:
$$Z_1=(1:1:1:0),\;
Z_2=(-1:-1:1:0),\;
Z_3=(1:-1:1:0),\;
Z_4=(-1:1:1:0).$$
We denote the union of these points by $Z$.
\begin{figure}[h!]
    \centering
\begin{tikzpicture}[line cap=round,line join=round,>=triangle 45,x=1cm,y=1cm,scale=0.8]
\clip(-6,-8) rectangle (10,10);
\draw [line width=1.5pt] ([shift=(-70:8)]0,0) arc (-70:110:8);
\draw [line width=1.5pt] (-4,0) -- (6,0); 
\draw [dotted,line width=1.5pt] (-4,3) -- (5,3);
\draw [dotted,line width=1.5pt] (-4,-3) -- (5,-3);
\draw [line width=1.5pt] (0,-4) -- (0,6); 
\draw [dotted,line width=1.5pt] (3,-4) -- (3,5); 
\draw [dotted,line width=1.5pt] (-3,-4) -- (-3,5); 
\draw [dotted,line width=1.5pt] (-4,-4) -- (5,5);
\draw [dotted,line width=1.5pt] (-4,4) -- (5,-5);
\draw [dashed,line width=1.5pt] (-4,-1) -- (2,5);
\draw [dashed,line width=1.5pt] (-1,4) -- (5,-2);
\draw [dashed,line width=1.5pt] (-1,-4) -- (5,2);
\draw [dashed,line width=1.5pt] (-4,1) -- (2,-5);
\begin{scriptsize}
\draw (0,0) circle (2.5pt);
\draw (3,0) circle (2.5pt);
\draw (-3,0) circle (2.5pt);
\draw (0,3) circle (2.5pt);
\draw (3,3) circle (2.5pt);
\draw (-3,3) circle (2.5pt);
\draw (0,-3) circle (2.5pt);
\draw (3,-3) circle (2.5pt);
\draw (-3,-3) circle (2.5pt);
\draw (45:8) circle (2.5pt);
\draw (-45:8) circle (2.5pt);
\draw (90:8) circle (2.5pt);
\draw (0:8) circle (2.5pt);
\draw (-0.35,0.8) node {$P_{15}$};
\draw (3-0.35,0.8) node {$P_{3}$};
\draw (-3-0.35,0.8) node {$P_{4}$};
\draw (-0.35,3.8) node {$P_{7}$};
\draw (3-0.35,3.6) node {$Z_1$};
\draw (-3+0.35,3.6) node {$Z_4$};
\draw (-0.35,0.8-3) node {$P_{8}$};
\draw (-3-0.35,0.6-3) node {$Z_2$};
\draw (3+0.35,0.6-3) node {$Z_3$};
\draw (90:7) node {$P_{14}$};
\draw (0:7) node {$P_{13}$};
\draw (45:7.5) node {$P_{1}$};
\draw (-45:7.5) node {$P_{2}$};
\end{scriptsize}
\end{tikzpicture}    
    \caption{$X\cup Y\cup Z$ with curved line $z=0$ at infinity}
    \label{fig: w=0 plane2}
\end{figure}
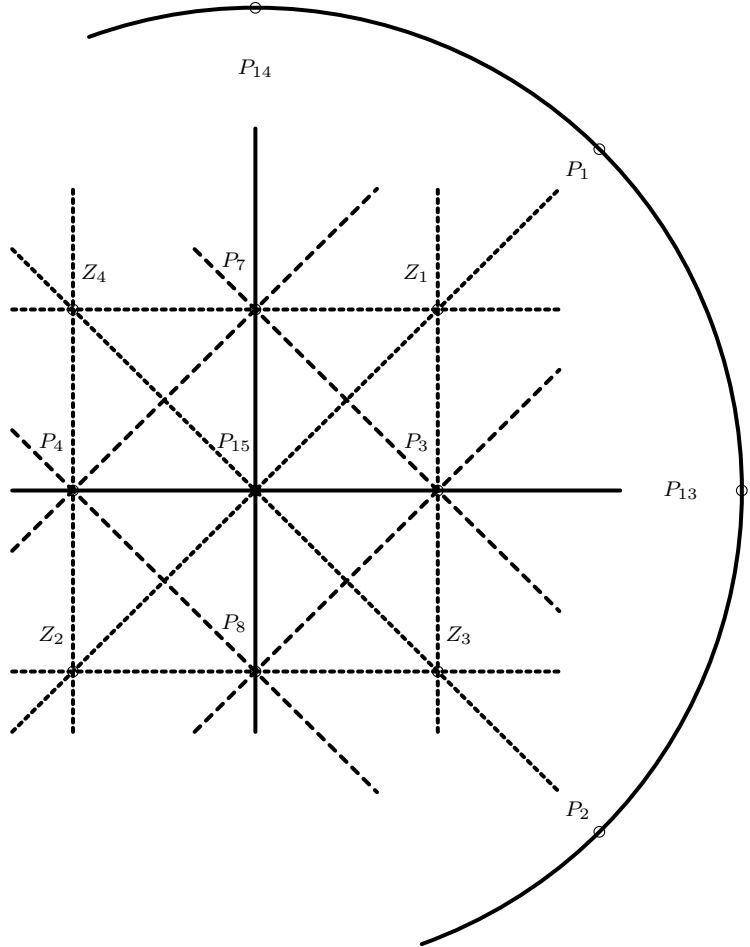
It is convenient to visualize all three sets of points $X$, $Y$, and $Z$, taking in $H$ the line $z=0$ as the line at infinity, see Figure \ref{fig: w=0 plane2}. 
The six lines joining pairs of points in $Z$, indicated in Figure \ref{fig: w=0 plane2} as dotted lines, form a complete quadrangle which we denote by $\varphi$.

\begin{notation}
For a curve $C\subset H$ we write $\nu(C)=(a;b,c,d)$, if 
$\deg(C)= a$ and $\mult_{P_i}C= b$ for $i\in X$, 
$\mult_{P_i}C= c$ for $i\in Y$, 
$\mult_{Z_i}C= d$ for $i=1,\ldots,4$; 
and we write $\nu(C)\cong(a;b,c,d)$, if 
$\deg(C)\leq a$ and $\mult_{P_i}C\geq b$ for $i\in X$, 
$\mult_{P_i}C\geq c$ for $i\in Y$, 
$\mult_{Z_i}C\geq d$ for $i=1,\ldots,4$.     
\end{notation}
In particular we have
\begin{equation}\label{eq: nu vectors}
\nu(\sigma)=(4;2,0,0),\; \nu(\tau)=(3;1,2,0),\; \nu(\varphi)=(6;1,2,3).    
\end{equation}

Since the Waldschmidt constant is a limit, it is enough to study a subsequence of values of $m$. For our approach it is convenient to assume that $m=18p$ for some odd number $p$. 

Let us now assume that $S$ is a surface in $\P^3$ of degree $d$, which vanishes to order at least $18p$ at each point of $\cF_4$ and such that
\begin{equation}\label{eq: lower WC}
    \frac{d}{18p}<\frac{8}{3}.
\end{equation}
This implies $d\leq 48p-1$.

We want to show that $S$ must contain $H$ as a component. Then $S$ would clearly contain all 12 planes in the one orbit of the $G$ action. Since the combinatorics are the same in both orbits, it follows that $S$ must also contain any plane in the $\H$ arrangement. Let us assume to the contrary that this is not the case, so that the intersection of $S$ and $H$ is some curve $\Gamma$ of degree $d$ and we have 
\begin{equation}\label{eq: nu of Gamma}
\nu(\Gamma)\cong(48p-1;18p,18p,3p+1),    
\end{equation}
with the last entry justified by the following Lemma.
\begin{lemma}\label{lem: vanishing on 3 point lines}
    The surface $S$ vanishes along any component of $\sigma$ to order at least $3p+1$.
\end{lemma}
\proof
To fix the notation, let $\ell$ be the line passing through $P_1, P_3$ and $P_8$. Let $\widetilde{H}$ be a general plane containing $\ell$ and let $\widetilde{\Gamma}$ be the trace of $S$ in $\widetilde{H}$. We are looking for the least integer $a$ such that $\ell$ is not forced to be a component of $\widetilde{\Gamma}$  B\'ezout's theorem, i.e., such $a$ that the inequality
$$48p-1-a\geq 3(18p -a)$$
holds. It follows $a\geq 3p+1$, as $a$ is an integer. Since $\widetilde{H}$ is general, the result follows.
\endproof
The same argument works for any of the $32$ lines containing $3$ points from the $\cF_4$ configuration. The traces in $H$ of these lines (other than the configuration points) are exactly the points in $Z$. Thus $\Gamma$ must vanish there to order at least $3p+1$.

Now we apply Lemma \ref{lem: vanishing on 3 point lines} and remove the copies of $\sigma$ which are forced to appear. We obtain a new divisor $\Gamma'=\Gamma-(3p+1)\sigma$ with
\begin{equation}\label{eq: nu of Gamma'}
\nu(\Gamma')\cong(36p-5;12p-2,18p,3p+1).
\end{equation}
We turn now our attention to the components of $\tau$.
\begin{lemma}\label{lem: vanishing on 4 point lines}
    The curve $\Gamma'$ contains each component of $\tau$ with multiplicity at least $8p+1$.
\end{lemma}
\proof
The argument is similar to that of Lemma \ref{lem: vanishing on 3 point lines}. We are now looking for the least $a$ such that any component of $\tau$ is not forced to be a component of $\Gamma'$ by B\'ezout's theorem. So, we must have
$$
36p-5-a\geq 2(18p-a)+2(12p-2-a).
$$
 Keeping in mind that $a$ is an integer, this gives $a\geq 8p+1$.
\endproof
Thus we can reduce the divisor $\Gamma'$ further subtracting $(8p+1)\tau$. This gives a new divisor $\Gamma''=\Gamma'-(8p+1)\tau$ with
\begin{equation}\label{eq: nu of Gamma''}
\nu(\Gamma'')\cong(12p-8;4p-3,2p-2,3p+1).
\end{equation}
Now, using B\'ezout's Theorem as before, it is easy to check that $\Gamma''$ contains the divisor $2\varphi+3\sigma$ as component. For this divisor we have
\begin{equation}\label{eq: nu of 2phi+3sigma}
\nu(2\varphi+3\sigma)=(24;8,4,6).
\end{equation}
Subtracting this divisor from $\Gamma''$ we obtain a new divisor $\Gamma'''$ with 
$$\nu(\Gamma''')\cong(12(p-2)-8;4(p-2)-3,2(p-2)-2,3(p-2)+1),$$
which has the same form as $\Gamma''$. Taking into account that $p$ was assumed to be odd, we can subtract $2\varphi+3\sigma$ altogether $\frac{p-1}{2}$ times. Thus we arrive finally to the residual divisor $\Sigma$ with
$$\nu(\Sigma)\cong(4;1,0,4).$$
A plane curve of degree $d$ can have multiple points of multiplicity $d$ if and only if all these points lie on a single line. Since this is not the case for points in $Z$, we see immediately that $\Sigma$ cannot exist.

We have showed in this way that the assumption in \eqnref{eq: lower WC} was false. Therefore 
$$\alpha(I(\cF_4)^{(18p)})\geq 48p,$$
and hence 
$$\alphahat(I(\cF_4))\geq \frac{8}{3}.$$  This completes the proof of Theorem \ref{thm: main}.

\section{The $H_4$ configuration}\label{sec: H4}

\subsection{The $H_4$ configuration} Let $\varphi$ be the golden ratio.  The $H_4$ root system $H_4 \subset \R^4$ is a highly symmetric configuration of $120$ vectors in $\R^4$; they are the vertices of a 600-cell.  If we let $\varphi$ be the golden ratio, then these vectors can be taken to be 
\begin{itemize}
\item $(\pm 1,0,0,0)$ and its permutations (8 vectors),
\item $\frac{1}{2}(\pm 1,\pm 1,\pm 1,\pm 1)$, where the signs are chosen independently (16 vectors),
\item $\frac{1}{2}(\pm \varphi, \pm 1 ,\pm \varphi^{-1},0)$ and its \emph{even} permutations, where the signs are chosen independently (96 vectors).
\end{itemize}
The Weyl group $G:=W(H_4) \subset \GL(\R^4)$ is the subgroup generated by reflections orthogonal to the vectors.  It is the finite complex reflection group $G_{30}$ of order 14,400 in the Shephard-Todd classification.  It can be generated by the following matrices:
$$\begin{pmatrix}-1 & 0 & 0 & 0\\ 0 & 1 & 0 & 0 \\ 0 & 0 & 1 & 0 \\ 0 & 0 & 0 & 1\end{pmatrix}, \quad \begin{pmatrix}1 & 0 & 0 & 0\\ 0 & -1 & 0 & 0 \\ 0 & 0 & 1 & 0 \\ 0 & 0 & 0 & 1\end{pmatrix}, \quad \frac{1}{2} \begin{pmatrix} \varphi & \varphi-1 & -1 & 0\\ \varphi-1 & 1 & \varphi & 0\\ -1 & \varphi & 1-\varphi & 0 \\ 0 & 0 & 0 & 2\end{pmatrix},  \quad \begin{pmatrix} 0 & 0 & -1 & 0\\ 0 & 0 & 0 & 1\\ 1 & 0 & 0 & 0\\ 0 & -1 & 0 & 0 \end{pmatrix}.$$

Consider the point $p=(0,0,0,1)$.  The first three of the above matrices all stabilize $p$, and in fact they generate the stabilizer $G_p$ of $p$. This is itself the rank 3 complex reflection group $G_p\cong W(H_3)\cong A_5\times \Z_2$ of symmetries of the icosahedron.  It is the group $G_{23}$ of order 120 in the Shephard-Todd classification.  Elements of the stabilizer of $p$ are block diagonal matrices of the form $$\begin{pmatrix} A & 0 \\ 0 & 1\end{pmatrix}$$ where $A$ is a matrix in the image of a particular representation $A_5\times \Z_2\to \GL(\R^3)$ realizing $A_5\times \Z_2$ as the group of symmetries of the icosahedron (compare with \cite[\S 2.2]{Calvo2024}).  The orbit of $p$ under $G$ consists precisely of the 120 points in the root system.  The groups $G$ and $G_p$ both act on the polynomial ring $S = \C[x,y,z,w]$ by linear changes of variables. 

We can also projectivize the root system $H_4$ to get a collection of $60$ points $\cH_4 \subset \P^3$.    The point $\tilde p = [0:0:0:1]$ has orbit $\cH_4$ under the induced action of $G$.  To compute the Waldschmidt constant of the ideal $I_{\cH_4}$, we need to study which pairs $(d,m)$ admit homogeneous polynomials $F \in S=\C[x,y,z,w]$ of degree $d$ which vanish to order $m$ at all of the points in $\cH_4$.  The next easy lemma reduces this problem to studying invariant polynomials.  For this reason, the rest of our arguments will all deal exclusively with invariant polynomials.

\begin{lemma}
Suppose that $F\in S = \C[x,y,z,w]$ is a homogeneous polynomial of degree $d$ which vanishes to order at least $m$ at all $60$ points in the $\cH_4$ configuration.  Then there is a $G$-invariant polynomial of some degree $dk$ which vanishes to order at least $mk$ at every point of $\cH_4$.  Therefore, the Waldschmidt constant $\widehat\alpha(I_{\cH_4})$ can be computed by only considering invariant polynomials.
\end{lemma}
\begin{proof}
The polynomial $\prod_{g\in G} gF$ does the trick for $k = |G|$.
\end{proof}

\subsection{Fundamental invariants for $G$}
Since $G$ is a complex reflection group, it is known that the ring of invariants $S^G \subseteq S$ is a polynomial ring $$T:=S^G = \C[f_2,f_{12},f_{20},f_{30}].$$  Here $f_d$ is a homogeneous polynomial of degree $d$, and $f_2,f_{12},f_{20},f_{30}$ are algebraically independent.  

It takes some work to explicitly describe the polynomials $f_d$.  Since $G$ acts by orthogonal matrices, it preserves distance and the polynomial $$f_2 := x^2+y^2+z^2+w^2$$ is invariant.  It is the unique degree $2$ invariant up to scale.

The rest of the invariants $f_{12},f_{20},f_{30}$ are not uniquely determined.  For instance, the space $T_{12}$ of degree $12$ invariants is a pencil (spanned by $f_2^6$ and $f_{12}$), and $f_{12}$ can be taken to be any degree $12$ invariant which is not a multiple of $f_2^6$.  Since we are interested in vanishing at the point $p$, we can canonically define $f_{12}$ up to scale by demanding that $f_{12}$ vanish at $p$.  

To explicitly compute $f_{12}$ we first need to identify any invariant of degree 12 which is not a multiple of $f_2^6$. To do this, we can for instance use the Reynold's operator \begin{align*}R_G&:S\to S^G \\ R_G(f) &= \frac{1}{|G|} \sum_{g\in G} gf\end{align*} to produce an invariant polynomial $$g_{12} = c_{12}R_G(x^{12})$$ where the constant $c_{12}=64/29925$ is chosen to normalize the coefficient of $x^{12}$ to $1$.  Then $$f_{12} := c_2'(f_2^6 - g_{12})$$ vanishes to order $2$ at $p$.  The coefficient $c_2' = 19/4$ is chosen to make the coefficient of $x^2w^{10}$ be $1$.  We refer the reader to the Appendix for Macaulay2 scripts that can be used to verify the computations in this section.

\begin{remark}
The order of vanishing of any invariant $f\in T$ at $p$ is even.  For instance, the group $G$ contains a transformation which acts on the variables by $x\mapsto -x$ and $y\mapsto y$, $z\mapsto z$, $w\mapsto w$.  Therefore every occurrence of $x$ in a monomial of $f$ has an even power.  The same conclusion holds for each other variable.  Then the order of vanishing at $p$ is measured by looking at the lowest degree monomial in $f(x,y,z,1)$, which is even.  See Proposition \ref{prop-leading} for a much more powerful statement.
\end{remark}

Next we compute a polynomial $f_{20}$.  The space $T_{20}$ has dimension $3$, and it can be computed that there is a unique invariant $f_{20}$ of degree $20$ which vanishes to order $4$ at $p$.  As in the previous case, we can first find any invariant $g_{20}$ of degree $20$ which is not a linear combination of $f_2^{10}$ and $f_2^4 f_{12}$.  Because $f_2$ is an invariant, the differential operator corresponding to $f_2$ takes invariant polynomials to invariant polynomials.  This is just the Laplacian $$\Delta := \frac{\partial^2}{\partial x^2}+\frac{\partial^2}{\partial y^2}+\frac{\partial^2}{\partial w^2}+\frac{\partial^2}{\partial z^2}.$$  Therefore, $$g_{20}:= \frac{1}{120}\Delta^2(f_{12}^2)$$ is an invariant of degree 20 (normalized so that the coefficient of $x^{20}$ is $1$), and it can be checked that it is independent of $f_2^{10}$ and $f_2^4f_{12}$.  It can then be computed that $$f_{20}:=\frac{5}{968}(-3f_2^{10}-224f_2^4f_{12}+3g_{20})$$ is the unique invariant of degree $20$ that vanishes to order $4$.  The coefficient of $x^4w^{16}$ is $1$.

Similar considerations lead us to define $$g_{30} := \frac{1}{42}\Delta(f_{12}f_{20})$$ and $$f_{30} := \frac{1}{40}(-15 f_2^5f_{20} - 6  f_2^3  f_{12}^2 + 21  g_{30});$$
it vanishes at $p$ to order $6$ (this property alone does not identify $f_{30}$ canonically, although $f_{30}$ also satisfies another property that identifies it canonically; this will be explained in the next subsection).

There are two additional invariants that are especially important to introduce before proceeding.  There is a unique invariant $f_{24}$ of degree $24$ that vanishes to order $6$ at $p$.  It is given by $$f_{24} = -\frac{1}{3}(f_{12}^2-f_2^2f_{20}).$$  There is also a unique invariant $f_{36}$ of degree $36$ that vanishes to order $10$ at $p$.  It has equation $$f_{36} = \frac{5}{27}(f_{12}^3-3f_2^2f_{12}f_{20}+2f_2^3f_{30}).$$

\begin{remark}
We will see that the polynomial $f_{36}$ of degree $36$ which vanishes to order $10$ computes the Waldschmidt constant $\widehat\alpha(I_{\cH_4}) = \frac{36}{10}$.
\end{remark}

\subsection{Fundamental invariants for $G_p$}

Similarly, the stabilizer $G_p$ also acts on $S$.  The block diagonal form of the matrices in $G_p$ shows that $G_p$ stabilizes $w$ and acts on $x,y,z$  by the symmetries of the icosahedron.  Accordingly, the ring of invariants $S^{G_p} \subseteq S$ is a polynomial ring $S^{G_p} = \C[s_2,s_6,s_{10},w]$, where $s_d(x,y,z)$ is a homogeneous polynomial in $x,y,z$ of degree $d$.  The polynomials $s_2,s_6,s_{10},w$ are algebraically independent.  As with $G$, the forms $s_6$ and $s_{10}$ are not uniquely determined; we will make a choice of them below.
 
We can use $G$-invariant polynomials which vanish at $p$ to construct $G_p$-invariant polynomials.

\begin{proposition}\label{prop-leading}
Let $f\in T$ be a $G$-invariant polynomial of degree $d$ which vanishes to order $m$ at $p$.  Write $$f(x,y,z,w) =  h_m(x,y,z)w^{d-m}+h_{m+1}(x,y,z)w^{d-m-1}+\cdots$$ where $h_d$ is homogeneous of degree $d$.  Then the leading term $h_m(x,y,z)$ is $G_p$-invariant.  Hence it is a polynomial in the fundamental invariants $s_2,s_6,s_{10}$.

Equivalently, if $I = (x,y,z)$ is the homogeneous ideal of $p$, then $G_p$ acts on $I^m/I^{m+1}$ and the element $\overline{f} \in I^m/I^{m+1}$ is $G_p$-invariant.
\end{proposition}
\begin{proof}
The group $G_p$ stabilizes $I$ and its powers $I^m$, so it acts on $I^m/I^{m+1}$.  Since $G$ stabilizes $f$, so does $G_p$, and therefore $G_p$ stabilizes $\overline f$.  From the first point of view, this means $G_p$ stabilizes $h_m(x,y,z)w^{d-m}$.  But $G_p$ also stabilizes $w$ and acts on $x,y,z$ without involving $w$, so $G_p$ also stabilizes $h_m(x,y,z)$.
\end{proof}

Now we make definitions of the fundamental invariants $s_2,s_6,s_{10}$ for $G_p$ that are compatible with our definitions of the fundamental invariants $f_2,f_{12},f_{20},f_{30}$ for $G$.  Recall that $f_{12}$, $f_{24}$, and $f_{36}$ vanish to orders $2,6$, and $10$, respectively.  We define homogeneous polynomials $s_d(x,y,z)$ by the congruences \begin{align*}s_2 &\equiv  f_{12}(x,y,z,1)  \pmod {I^3}\\
 s_6 &\equiv f_{24}(x,y,z,1) \pmod {I^7}\\
 s_{10} &\equiv f_{36}(x,y,z,1) \pmod {I^{11}}.
 \end{align*}
 It can be verified that $s_2^3,s_6$ are linearly independent and $s_2^5,s_2^2s_6,s_{10}$ are linearly independent, and therefore that $s_2,s_6,s_{10}$ can be taken to be the fundamental invariants for $G_p$.

Given an $f\in T$ of degree $d$ which vanishes to order $k$, we write $\overline f\in I^k/I^{k+1}$ for its image in $I^k/I^{k+1}$.  Observe that we have an isomorphism of the invariants $$(I^k/I^{k+1})_d^{G_p} \cong \C[s_2,s_6,s_{10}]_k \cdot w^{d-k},$$ so that $\overline f$ can be written as $w^{d-k}$ times a homogeneous form in $s_2,s_6,s_{10}$ of degree $k$.  We compute the images of our invariants of $G$:
\begin{align*}
\overline f_2 & = w^2\\
\overline f_{12} &= s_2w^{10}\\
\overline f_{20} &= s_2^2 w^{16}\\
\overline f_{24} &= s_6 w^{18}\\
\overline f_{30} &= s_2^3 w^{24}\\
\overline f_{36} &= s_{10} w^{26}.
\end{align*}
The power of $w$ is a convenient placeholder that lets us remember what degree the form came from.
\begin{remark}
The invariant $f_{30}$ is canonically determined by the properties that it vanishes to order $6$ at $p$ and that $\overline f_{30} = s_2^3 w^{24}$.  Since $f_2^3 f_{24}$ also vanishes to order $6$, adding some amount of $f_2^3f_{24}$ to $f_{30}$  would yield a degree $30$ invariant $g_{30}$ such that $\overline g_{30}$ is a linear combination of $s_6w^{24}$ and $s_2^3 w^{24}$.
\end{remark}

\subsection{The associated graded ring} A more ambitious problem than the computation of the Waldschmidt constant $\widehat \alpha(I_{\cH_4})$ is the computation of the dimension $$\dim (T_d\cap I^m)$$ of the space of degree $d$ invariant forms which vanish to multiplicity $m$ at $p$ for all possible pairs $(d,m)$.  This is what we will now do.

For fixed $d$ and large $m$, we have $T_d\cap I^m = 0$.   Therefore for any $m_0\geq 0$ we have $$\dim (T_d \cap I^{m_0}) = \sum_{m\geq m_0} \dim\left( \frac{T_d\cap I^{m}}{T_d\cap I^{m+1}}\right).$$ The dimensions on the right hand side can be viewed as the number of independent linear conditions it is for a form of degree $d$ and multiplicity $m$ to vanish to order $m+1$.  Computing the dimensions of $T_d\cap I^m$ is therefore an equivalent problem to computing the dimensions of the spaces $(T_d\cap I^m)/(T_d\cap I^{m+1})$.  We can package this information for all $d$ and $m$ into an associated graded ring.

\begin{definition}
We define a $\C$-algebra $$R = \bigoplus_{m\geq 0} \frac{T\cap I^m}{T\cap I^{m+1}}.$$ Observe that $R$ is graded both by the degree $d$ of a form and the order of vanishing $m$.  An element of $(T_d\cap I^m) / (T_d \cap I^{m+1})$ has \emph{bidegree} $(d,m)$.
\end{definition}

Our main result will show that $R$ is finitely generated and identify an explicit list of generators.  Thus, its Hilbert series and the dimensions of its graded pieces can be effetively computed.  The next result enables easy computation inside the ring $R$.

\begin{proposition}
Give the ring $\C[s_2,s_6,s_{10},w]$ a bigrading by giving $s_d$ bidegree $(d,d)$ and $w$ bidegree $(1,0)$.  Then the assignment $f \mapsto \overline f$ defines a bidegree preserving inclusion of rings $R\to \C[s_2,s_6,s_{10},w]$.
\end{proposition}
\begin{proof}
The assignment $f\mapsto \overline f$ gives inclusions \begin{align*}\frac{T\cap I^m}{T\cap I^{m+1}} &\to \frac{I^{m}}{I^{m+1}}
\end{align*} defining an inclusion of associated graded rings $$R\to \bigoplus_{m\geq 0} \frac{I^{m}}{I^{m+1}}.$$ The image lies in the space of $G_p$-invariants, which can be identified with $\C[s_2,s_6,s_{10},w]$. 
\end{proof}

\subsection{Generators of $R$}\label{ssec: Rgens}

We have already discussed several $G$-invariant polynomials $f$ with large orders of vanishing at $p$.  Each of their images $\overline f$ gives a generator of the ring $R$.  In total we will need 12 generators to generate $R$.  We present the list of all our generators in Table \ref{tab-1}.  In this table we list a bidegree $(d,m)$, an element $f\in T_d\cap I^m$, and its image $\overline f \in \C[s_2,s_6,s_{10},w]$.  The fact that $f\in I_m$ and the expression for $\overline f$ can be readily verified with Macaulay2.  We only record $f$ and $\overline f$ up to scale, as the particular multiple is not important.
\begin{table}$$\renewcommand\arraystretch{1.3}\begin{array}{c|l|l}
(d,m) & f \textrm{ up to scale} & \overline f \textrm{ up to scale}\\\hline
(2,0) &f_2 & w^2\\ 
(12,2) & f_{12} & s_2 w^{10}\\
(20,4) & f_{20} & s_2^2w^{16}\\
(24,6) & f_{12}^2-f_2^2f_{20} & s_6 w^{18}\\
(30,6) & f_{30} & s_2^3w^{24}\\
(32,8) & f_{12}f_{20}-f_2f_{30} & s_2s_6w^{24}\\
(36,10) &  f_{12}^3-3f_2^2f_{12}f_{20}+2f_2^3f_{30} & s_{10}w^{26}\\
(42,10) & f_2f_{20}^2-f_{12}f_{30} & s_2^2s_6w^{32}\\
(44,12) &  f_{12}^2f_{20}+f_2^2f_{20}^2-2f_2f_{12}f_{30} & (3 s_2s_{10}-5s_6^2)w^{32}\\
(54,14) &2f_2f_{12}f_{20}^2-f_{12}^2f_{30}-f_2^2f_{20}f_{30} & (6s_2^2s_{10}-5s_2s_6^2)w^{40}\\
(60,16) & f_{20}^3-f_{30}^2 & (4s_2^3s_{10}-5s_2^2s_6^2)w^{44}\\
(66,18) & 3f_2f_{12}^2f_{20}^2-f_2^3f_{20}^3-f_{12}^3f_{30} & (9s_2s_6s_{10}-10s_6^3)w^{48}\\ &\qquad  -3f_2^2f_{12}f_{20}f_{30}+2f_2^3f_{30}^2
\end{array}$$
\caption{The 12 generators $\overline f$ of $R$.}\label{tab-1}\end{table}

\begin{theorem}\label{thm-H4Main}
The associated graded ring $$R = \bigoplus_{m\geq 0} \frac{T\cap I^m}{T\cap I^{m+1}}$$ is generated by the 12 elements $\overline f$ listed in Table \ref{tab-1}.  The bigraded Hilbert series $$HS(u,v):= \sum_{d,m} \dim\left(\frac{T_d\cap I^m}{T_d\cap I^{m+1}}\right)u^dv^m$$ is $$\tfrac{(1+u^{12}v^2+u^{24}v^6+u^{30}v^6+u^{32}v^8-u^{32}v^6+u^{42}v^{10}-u^{44}v^{10}+u^{54}v^{14}-u^{56}v^{14}+u^{66}v^{18}-u^{66}v^{16}-u^{68}v^{18}-u^{74}v^{18}-u^{86}v^{22}-u^{98}v^{24})}{(1-u^{60}v^{16})(1-u^{44}v^{12})(1-u^{36}v^{10})(1-u^{20}v^4)(1-u^2)}.$$
\end{theorem}
\begin{proof}
Let $R'\subseteq R \subset \C[s_2,s_6,s_{10},w]$ be the subalgebra generated by the 12 elements $\overline f$ in Table \ref{tab-1}.  Since $s_2,s_6,s_{10},w$ are algebraically independent, we may treat $s_2,s_6,s_{10},w$ as indeterminates where $s_d$ has bidegree $(d,d)$ and $w$ has bidegree $(1,0)$.  Then $R'$ is a finitely generated $\C$-algebra, and it can be written as a quotient of a polynomial algebra generated by 12 indeterminates of the bidgrees $(d,m)$ appearing in Table \ref{tab-1}.  Then Macaulay2 can quickly compute the Hilbert series $$HS'(u,v):= \sum_{d,m} \dim R'_{d,m} u^dv^m$$ of $R'$, and it produces the rational function in the statement of the theorem for an answer.

Because $R'_{d,m} \subseteq R_{d,m}$ for all $d,m$, we have $\dim R'_{d,m} \leq \dim R_{d,m}$ for all $d,m$.  If we put $v = 1$ in the Hilbert series $HS(u,v)$ of $R$, the coefficient of $u^d$ will be $\dim T_d$ since $$\dim T_d = \sum_{m\geq 0} \dim\left(\frac{T_d\cap I^m}{T_d\cap I^{m+1}}\right) = \sum_{m\geq 0} \dim R_{d,m}.$$ Since $\dim T_d$ is the number of degree $d$ monomials in $f_2,f_{12},f_{20},f_{30}$ we get $$HS(u,1) = \frac{1}{(1-u^{30})(1-u^{20})(1-u^{12})(1-u^2)}.$$
On the other hand, if we put $v=1$ into the Hilbert series $HS'(u,v)$ of $R'$, we likewise find $$HS'(u,1) = \frac{1}{(1-u^{30})(1-u^{20})(1-u^{12})(1-u^2)}.$$  It follows that $$\sum_{m\geq 0} \dim R'_{d,m} = \dim T_d = \sum_{m\geq 0} \dim R_{d,m}.$$ Then since $\dim R_{d,m}'\leq \dim R_{d,m}$ for all $d,m$ we conclude that $\dim R_{d,m}' = \dim R_{d,m}$ for all $d,m$.  Therefore $R' = R$ and $HS(u,v) = HS'(u,v)$.
\end{proof}

\begin{corollary}\label{cor-waldschmidt}
The Waldschmidt constant of the configuration $\cH_4$ of 60 points in $\P^3$ is $$\widehat \alpha(\cH_4) = \frac{18}{5},$$ computed by the invariant $f_{36}$ of degree $36$ which vanishes to order $10$ along the points of $\cH_4$.
\end{corollary}
\begin{proof}
The positive real cone spanned by the bidegrees $(d,m)$ of the generators of $R$ is spanned by $(2,0)$ and $(36,10)$.  The bidegree of any element of $R$ then lies in this cone, so $T_d \cap I^m$ is empty if $\frac{d}{m}<\frac{36}{10}$.  Thus $\widehat\alpha(\cH_4)\geq \frac{36}{10}$, and the existence of $f_{36}$ gives the other inequality.\end{proof}

While Theorem \ref{thm-H4Main} gives an explicit generating function for the dimension of $(T_d\cap I^m)/(T_d\cap I^{m+1})$, we can also easily visualize these dimensions for relatively small $d$ and $m$.  For a fixed $m$, if $d$ is large enough then the inclusion $$\frac{T_d\cap I^m}{T_d\cap I^{m+1}} \to \C[s_2,s_6,s_{10},w]_{d,m}=\C[s_2,s_6,s_{10}]_m \cdot w^{d-m}$$ becomes an isomorphism.  This is true because the invariants $f_2,f_{12},f_{24},f_{36}$ map to $w^2,s_2w^{10},s_6w^{18}$, and $s_{10}w^{26}$, respectively.  Thus if $d$ is large enough, it is $\dim \C[s_2,s_6,s_{10}]_m$ independent conditions on forms in $T_d\cap I^m$ to vanish to order $m+1$.  We also have inclusions $$\frac{T_d\cap I^m}{T_{d}\cap I^{m+1}} \to \frac{T_{d+2} \cap I^m}{T_{d+2}\cap I^{m+1}}$$ given by multiplication by $f_2$, so for fixed $m$ and even $d$ the dimension $\dim R_{d,m}$ increases with $d$.  Whenever $\dim R_{d,m} > \dim R_{d-2,m},$ the difference $\dim R_{d,m}-\dim R_{d-2,m}$ measures the number of additional independent conditions it is for a degree $d$ form to vanish to order $m+1$ compared to the degree $d-2$ case.  For each $m$, we can record the degrees $d$ where this difference is positive, and we record $d$ a total number of $\dim R_{d,m}-\dim R_{d-2,m}$ times.  We know we can stop once $\dim R_{d,m} = \dim \C[s_2,s_6,s_{10}]_m.$  This process generates Table 2, and the next corollary follows.

\begin{corollary}\label{cor-conditions}
The codimension of $T_d \cap I^m$ in $T_d$ is the number of times an integer $\leq d$ appears on the right side of Table \ref{tab-2} in rows $0,2,\ldots,m-2$.

The dimension of $T_d\cap I^m$ is the number of times an integer $\leq d$ appears on the right side of Table \ref{tab-2} in rows $m,m+2,\ldots.$
\end{corollary}

\begin{example}
From Corollary \ref{cor-conditions} and Table \ref{tab-2}, we see that it is 25 independent linear conditions on forms of degree $72$ to vanish to order $20$ at $p$.  Since $\dim T_{72} = 26$,  there is up to scale a unique form of degree $72$ vanishing to order $20$ at $p$, namely $f_{36}^2$.
\end{example}

\begin{table}$$\renewcommand\arraystretch{1.2}
\begin{array}{c|ccccccccccccc}
m &\multicolumn{13}{c}{\textrm{degrees }d\textrm{ with }\dim R_{d,m}- \dim R_{d-2,m} >0} \\ \hline
0 & 0\\
2 & \underline{12}\\
4 & \underline{20}\\
6 & \underline{24} & \underline{30}\\
8 & \underline{32} & 40\\
10 & \underline{36} & \underline{42} & 50\\
12 & \underline{44} & 48 & 52 & 60\\
14 & \underline{54} & 56 & 62 & 70\\
16 & \underline{60} & 60 & 64 & 72 & 80\\
18 & \underline{66} & 68 & 72 & 74 & 82 & 90\\
20 &  72 &76 &78 &80 &84 &92 &100\\
22  &80 &84 &84 &86 &90 &94 &102 &110 \\
24  &88 &90 &92 &92 &96 &100 &104 &112 &120 \\
26  &96 &96 &98 &100 &102 &106 &110 &114 &122 &130 \\
28  &102 &104 &104 &108 &108 &112 &116 &120 &124 &132 &140 \\
30  &108 &110 &112 &114 &114 &116 &118 &122 &126 &130 &134 &142 &150 \\ 
\end{array}$$
\caption{For fixed $m$, we list the degrees $d$ where $\dim R_{d,m} - \dim R_{d-2,m}>0$. If this difference is more than $1$, we list $d$ that many times. See Corollary \ref{cor-conditions}.  Degrees corresponding to generators of $R$ are underlined.}\label{tab-2}\end{table}

\section*{Appendix: Macaulay2 code}

In this section we include Macaulay2 scripts that are used to verify the details of the computations in Section \ref{sec: H4}.  The first scripts input the group $G$ of order 14,400 into M2 and construct the fundamental invariant polynomials $f_2,f_{12},f_{20},f_{30}$.

\begin{verbatim}
loadPackage "InvariantRing"

--base ring
K = toField(QQ[c]/(c^2-5/4))
S = K[x,y,z,w]
a = c+1/2
I = ideal(x,y,z)

--group generators
m1 = matrix{{-1,0,0,0},{0,1,0,0},{0,0,1,0},{0,0,0,1}}
m2 = matrix{{1,0,0,0},{0,-1,0,0},{0,0,1,0},{0,0,0,1}}
m3 = -1/2 * matrix{{-a,1-a,1,0},{1-a,-1,-a,0},{1,-a,a-1,0},{0,0,0,-2}}
m4 = matrix{ {0,0,-1,0},{0,0,0,1},{1,0,0,0},{0,-1,0,0}}

--construct the group of order 14,400
glist = {m1,m2,m3,m4};
G = group finiteAction(glist,S); -- matrix group
#G  --verify 14,400 elements

--construct the fundamental invariants f2, f12, f20, f30
f2 = x^2+y^2+z^2+w^2;
mySum = 0_S;
for i from 0 to 14399 do (
    mySum = mySum + (G_i_(0,0)*x+G_i_(0,1)*y+G_i_(0,2)*z+G_i_(0,3)*w)^12);
g12=mySum*64/29925;
f12 = (f2^6-g12)*19/4;
g20 = diff(f2^2,f12^2)/120;
f20 = (-3*f2^10-224*f2^4*f12+3*g20)*5/968; --degree 20, vanish 4
g30 = diff(f2,f12*f20)/42;
f30 = (-15*f2^5*f20-6*f2^3*f12^2+21*g30)/40; --degree 30, vanish 6
\end{verbatim}

Next we introduce the additional $G$-invariant polynomials $f_d$ with high orders of vanishing which are necessary to generate the ring $R$ in Section \ref{ssec: Rgens}.

\begin{verbatim}
--G-invariant polynomials with high order of vanishing
f24 = (f12^2-f2^2*f20)/-3; 
f32 = (f12*f20-f2*f30)*-2/3; 
f36 = (f12^3-3*f2^2*f12*f20+2*f2^3*f30)*5/27; 
f42 = (-f2*f20^2+f12*f30)*-2/3; 
f44 = (f12^2*f20+f2^2*f20^2-2*f2*f12*f30)*10/18; 
f54 = (2*f2*f12*f20^2-f12^2*f30-f2^2*f20*f30)*10/9; 
f60 = (f20^3-f30^2)*20/27; 
f66 = (-3*f2*f12^2*f20^2+f2^3*f20^3+f12^3*f30+3*f2^2*f12*f20*f30-2*f2^3*f30^2)*10/27; 
\end{verbatim}

We now introduce the fundamental invariants $s_2,s_6,s_{10}$ for the stablizer $G_p$ and check that the polynomials $f_d$ have images $\overline f_d$ as described in Table \ref{tab-1}.

\begin{verbatim}   
--fundamental generators of stabilizer invariants
use S;
s2 = sub(f12 % I^3,w=>1);
s6 = sub(f24 % I^7,w=>1);
s10 = sub(f36 % I^11,w=>1);

--hi is fi bar, the image of fi in the associated graded ring R
h2 = f2 % I^1;  --deg (2,0)
h12 = f12 % I^3; --deg (12,2)
h20 = f20 % I^5; --deg (20,4)
h24 = f24 % I^7; --deg (24,6)
h30 = f30 % I^7; --deg (30,6)
h32 = f32 % I^9; --deg (32,8)
h36 = f36 % I^11; --deg (36,10)
h42 = f42 % I^11; --deg (42,10)
h44 = f44 % I^13; --deg (44,12)
h54 = f54 % I^15; --deg (54,14)
h60 = f60 % I^17; --deg (60,16)
h66 = f66 % I^19; --deg (66,18)

--verify the expressions for fi bar given in Table 1
h2 == w^2
h12 == s2 * w^10
h20 == s2^2 * w^16
h24 == s6 * w^18
h30 == s2^3 * w^24
h32 == s2 * s6 * w^24
h36 == s10 * w^26
h42 == s2^2 * s6 * w^32
h44 == (-3*s2*s10+5*s6^2)*w^32
h54 == (-6*s2^2*s10+5*s2*s6^2)*w^40
h60 == (-4*s2^3*s10+5*s2^2*s6^2)*w^44
h66 == (-9*s2*s6*s10+10*s6^3)*w^48
\end{verbatim}

We next compute the Hilbert series of the ring $R'$ in Theorem \ref{thm-H4Main}.

\begin{verbatim}
stabAlg = K[t2,t6,t10,u,Degrees=>{{2,2},{6,6},{10,10},{1,0}}];
polyAlg = K[j2,j12,j20,j30,j24,j32,j36,j42,j44,j54,j60,j66,
    Degrees=>{{2,0},{12,2},{20,4},{30,6},{24,6},{32,8},{36,10},
        {42,10},{44,12},{54,14},{60,16},{66,18}}];
phi = map(stabAlg,polyAlg,
    {u^2,
    t2*u^10,
    t2^2*u^16,
    t2^3*u^24,
    t6*u^18,
    t2*t6*u^24,
    t10*u^26,
    t2^2*t6*u^32,
    (-3*t2*t10+5*t6^2)*u^32,
    (-6*t2^2*t10+5*t2*t6^2)*u^40,
    (-4*t2^3*t10+5*t2^2*t6^2)*u^44,
    (-9*t2*t6*t10+10*t6^3)*u^48});
J = ker phi;
R' = polyAlg/J;
reduceHilbert hilbertSeries(R')
\end{verbatim}

The output of the last command is the rational function appearing in Theorem \ref{thm-H4Main}. Finally, the following script generates the data in Table \ref{tab-2}.

\begin{verbatim}
condsRing = K[w2,w6,w10,Degrees=>{2,6,10}]
for l from 1 to 15 do(
    k = 2*l;
    << k << "  ";
    conds = hilbertFunction(k,condsRing);
    lastHilb = 0;
    d=12;      
    while lastHilb<conds do(
        currentHilb=hilbertFunction({d,k},R');  
        while(currentHilb>lastHilb) do(
            << "& " << d << " ";
            lastHilb=lastHilb+1;
        );
        d=d+2;
    );
    << "\\\\" << endl << flush;
)
\end{verbatim}

%\bibliographystyle{abbrv}
%\bibliography{psu}

\end{document}